\documentclass{amsart}

\newtheorem{theorem}{Theorem}[section]
\newtheorem{lemma}[theorem]{Lemma}
\newtheorem{corollary}[theorem]{Corollary}
\newtheorem{proposition}[theorem]{Proposition}

\theoremstyle{definition}

\newtheorem{example}[theorem]{Example}

\theoremstyle{remark}

\numberwithin{equation}{section}

\newcommand{\abs}[1]{\lvert#1\rvert}

\newcommand{\n}[1]{\Vert#1\Vert}
\newcommand{\C}{\mathbb C}
\newcommand{\D}{\mathbb D}
\newcommand{\R}{\mathbb R}
\newcommand{\U}{\mathbb U}

\begin{document}

\title{Semigroups of Composition Operators on Hardy Spaces of the half-plane}

\author{Athanasios G. Arvanitidis }
\address{Department of Mathematics, University of Thessaloniki,
         54124 Thessaloniki, Greece}
\email{arvanit@math.auth.gr}

\thanks{The author was supported with a graduate fellowship from the Alexander S. Onassis
Foundation.}


\date{\today}



\begin{abstract}
We identify the semigroups consisting of bounded composition
operators on the Hardy spaces $H^p(\U)$ of the upper half-plane.
We show that any such semigroup is strongly continuous on
$H^p(\U)$ but not uniformly continuous and we identify the
infinitesimal generator.
\end{abstract}

\maketitle

\markboth{ATHANASIOS G. ARVANITIDIS}{SEMIGROUPS OF COMPOSITION
OPERATORS ON HARDY SPACES}

\section{Introduction}

Let $\U=\{z\in\mathbb{C}:\textrm{Im}z>0\}$ denote the upper half
of the complex plane. The Hardy space $H^{p}(\U)$, $0 < p<\infty,$
is the space of analytic functions $f:\U\rightarrow\mathbb{C}$ for
which
$$
\n{f}_{p}=\sup_{y>0}\Bigl(\int_{-\infty}^{\infty}
\abs{f(x+iy)}^{p}\,dx\Bigr)^{\frac{1}{p}}<\infty.
$$

For $1\leq p\leq\infty$, the spaces $H^{p}(\U)$ are Banach spaces
and $H^{2}(\U)$ is a Hilbert space. Furthermore for  $f\in
H^{p}(\U)$, $1\leq p<\infty$, the limit
$\lim_{y\rightarrow0}f(x+iy)$ exists for almost every $x \in
\mathbb{R}$ and we may define the boundary function on
$\mathbb{R}$, again denoted by $f$, as
$$
f(x)=\lim_{y\rightarrow0}f(x+iy).
$$
This function is p-integrable and
\begin{equation*}
\n{f}_{p}^{p}=\int_{-\infty}^{\infty}\abs{f(x)}^{p}\,dx.
\end{equation*}
For more details on Hardy spaces see \cite{Du}, \cite{Ga}.

Let $\mathcal{H}(\U)$ denote the space of all analytic functions
on $\U$ and $\phi:\U\to\U$ be  analytic. The composition operator
induced by $\phi$ is defined by
$$
C_{\phi}(f)=f\circ\phi, \quad f\in \mathcal{H}(\U).
$$
If $X$ is a linear subspace of $\mathcal{H}(\U)$ which is a Banach
space under a norm $\n{\,\,}_X$ we can consider the restriction of
$C_{\phi}$ on $X$. The question arises whether $C_{\phi}$ acts as
a bounded operator on $X$, that is if $f\circ\phi\in X$ for each
$f\in X$ and if that is the case if $\n{f\circ\phi}_X\leq
C\n{f}_X$ for a constant $C$. We will not consider this question
here, but  we mention that in contrast to the case of the Hardy
spaces of the unit disc, there are self-maps $\phi$ of $\U$ which
do not induce bounded composition operators on the Hardy spaces
$H^p(\U)$. More details for this will be presented in the next
section.

Suppose now that $\{\phi_{t}:t\geq0\}$ is a one-parameter
semigroup under composition of analytic self-maps of $\U$, that
is:

\vspace{0.2cm}\noindent
(1) $\phi_{0}(z)\equiv z$, \, \, the identity map of $\U$. \\
(2) $\phi_{t+s}=\phi_{t}\circ\phi_{s}$ \, \, for $t, s\geq0$. \\
(3) The map \, $(t,z)\to\phi_{t}(z)$ \, is jointly continuous on
\, $[0,+\infty)\times \U$.

\vspace{0.2cm}\noindent Then the induced maps
$$
T_t(f)=f\circ\phi_t
$$
form a semigroup of linear transformations on $\mathcal{H}(\U)$.
Semigroups of analytic functions on the half-plane and on the disc
was first studied by E. Berkson and H. Porta in \cite{BP}, where
they also prove the strong continuity of the induced semigroups of
composition operators on the Hardy spaces of the disc. Here we are
concerned with analogous questions on the Hardy spaces of the
half-plane.

Considering $\{T_t\}$ on a $H^p(\U)$ space we show that if one
$T_t, \ t>0$ is a bounded operator then all $T_t$ are bounded,
giving also examples of unbounded semigroups $\{T_t\}$. Assuming
further that $1\leq p<\infty$ and each $T_t$ is bounded on
$H^p(\U)$, we prove the strong continuity of $\{T_t\}$ on
$H^p(\U)$ and identify its infinitesimal generator.

A specific semigroup of composition operators was used in
\cite{AS} to study the Ces\`{a}ro operator and its adjoint on the
Hardy spaces of the half-plane.

\section{Preliminaries}

\subsection{Composition operators on $H^p(\U)$}
Recent results give characterizations of bounded composition
operators on $H^{p}(\U)$ spaces in terms of angular derivatives.
Next we recall the necessary definitions and tools for them.

Let $f:\U\to\C$ be an analytic function. If $f(z)\to c$, where
$c\in \C\cup\{\infty\}$, as $z=x+iy\to\infty$ through any sector
$$
T_u(\infty)=\{x+iy\in \U:|x|<uy\}, \quad u>0,
$$
we say that $c$ is the non-tangential limit of $f$ at $\infty$ and
we denote it by
$$
\angle \lim_{z\to \infty}f(z).
$$
Let $\phi:\U\to\U$ be an analytic function. The
Julia-Carath\'{e}odory theorem for the upper half-plane (see
\cite[Exercise 2.3.10]{CM}) says that
$$
\angle \lim_{z\to \infty}\frac{\phi(z)}{z} = \angle \lim_{z\to
\infty}\phi'(z)=\inf_{z\in\U}\frac{\textrm{Im}\phi(z)}{\textrm{Im}z}.
$$
From this it is clear that the limit
$$
\phi'(\infty):=\angle \lim_{z\to \infty}\frac{\phi(z)}{z},
$$
which we call it the angular derivative of $\phi$ at $\infty$,
always exists and it belongs to $[0, \ +\infty)$.

We consider also the conjugate function $\psi$ of $\phi$ on the
disc $\D=\{w\in\C:|w|<1\}$,
$$
\psi=\gamma^{-1}\circ\phi\circ\gamma:\D\to\D,
$$
where $\gamma(w)=i\frac{1+w}{1-w}$, a conformal map from $\D$ onto
$\U$, with inverse $\gamma^{-1}(z)=\frac{z-i}{z+i}$, $z\in\U$.
Similar we denote the non-tangential limit of $\psi$ at
$\zeta\in\partial\D$, if this exists, by
$$
\psi(\zeta):=\angle \lim_{w\to\zeta}\psi(w),
$$
that is the limit of $\psi(w)$ as $w\to \zeta$ through any sector
$$
S_a(\zeta)=\{w\in\D:|w-\zeta|<a(1-|w|)\}, \quad a>1.
$$
In particular if $\psi(1)=1$, we will use the angular derivative
of $\psi$ at 1,
$$
\psi'(1)=\angle \lim_{w\to 1}\frac{1-\psi(w)}{1-w},
$$
which is known (we see it also by Lemma \ref{equality_of_angular
derivatives}) that always exists (it may be $\infty$).

\begin{lemma}\label{non-tang conv}
Let $w\in\D$. Then $w\to 1$ non-tangentially if and only if
$\gamma(w)\to \infty$ non-tangentially.
\end{lemma}
\begin{proof}
It is clear that $w\to 1$ if and only if
$\gamma(w)=i\frac{1+w}{1-w}\to \infty$.  Let $a>1$, $w\in S_a(1)$
and $z=\gamma(w)=x+iy$. Since
\begin{align*}
\frac{|w-1|}{1-|w|}&=
\frac{|\gamma^{-1}(z)-1|}{1-|\gamma^{-1}(z)|}=
\frac{2}{|z+i|-|z-i|}\\
&=\frac{2(|z+i|+|z-i|)}{|z+i|^2-|z-i|^2}>\frac{4|x|}{(y+1)^2-(y-1)^2}=\frac{|x|}{y}
\end{align*}
we get that $z\in T_a(\infty)$. Conversely let $u>0$ and $z\in
T_u(\infty)$ with $y>1$. Then $\frac{|x|}{y}<u$ and it follows
that
$$
2(u+1)>\frac{2(|x|+y)}{y}=\frac{|x|+y+1+|x|+y-1}{y}>\frac{|z+i|+|z-i|}{|z+i|^2-|z-i|^2},
$$
thus by the above we get that $w\in S_{4(u+1)}(1)$ and the
conclusion follows.
\end{proof}

\begin{lemma}\label{equality_of_angular derivatives}
Let $\phi:\U\to\U$ be analytic and $\psi$ its conjugate on $\D$.
If $\psi(1)=1$, then
\begin{equation}
\psi'(1)=\frac{1}{\phi'(\infty)}.
\end{equation}
In particular, $\psi(1)=1$ and $\psi'(1)<\infty$ if and only if
$\phi'(\infty)>0$.
\end{lemma}
\begin{proof}
We have
\begin{align*}\label{ang_1}
\psi'(1)&=\angle\lim_{w\to
1}\frac{1-\psi(w)}{1-w}=\angle\lim_{w\to
1}\frac{(1-\psi(w))(1+w)}{(1+\psi(w))(1-w)}=\angle\lim_{w\to
1}\frac{\gamma(w)}{\gamma(\psi(w))}\nonumber\\
&=\angle\lim_{z\to
\infty}\frac{z}{\phi(z)}=\frac{1}{\phi'(\infty)}.
\end{align*}
If $\phi'(\infty)>0$, then by definition follows that $\angle
\lim_{z\to \infty}\phi(z)=\infty$, so
$$
\angle\lim_{w\to 1}\psi(w)=\angle\lim_{z\to
\infty}\gamma^{-1}(\phi(z))=1
$$
and furthermore $\psi'(1)<\infty$.
\end{proof}

Now let $0 < p<\infty$ and $\phi:\U\to\U$ be analytic. V. Matache
(\cite[Theorem 15]{Ma2}, \cite{Ma1}) showed that the induced
composition operator $C_{\phi}:H^{p}(\U)\to H^{p}(\U)$ is bounded
if and only if the conjugate function $\psi$ of $\phi$ has
$\psi(1)=1$ and $\psi'(1)<\infty$, i.e. if and only if
$\phi'(\infty)>0$. But the exact norm of $C_{\phi}$ was found by
S. Elliott and M. Jury in \cite{EJ} (see \cite[Corollary 3.5 and
Definition 2.4]{EJ}) and is
\begin{equation}\label{C estimate}
\n{C_{\phi}}=\phi'(\infty)^{-\frac{1}{p}}.
\end{equation}

Also an important role in the study of $C_{\phi}$ plays the
Denjoy-Wolff point of $\phi$. We recall that the Denjoy-Wolff
theorem for analytic self-maps of $\D$ \cite[Th. 2.51]{CM} through
the map $\gamma$ asserts that if $\phi$ is not the identity or an
elliptic automorphism, there is a point
$d\in\overline{\U}=\U\cup\R\cup\{\infty\}$ such that the sequence
of iterates $\phi_n\to d$ uniformly on compact subsets of $\U$. In
the case of elliptic automorphism $\phi$ has a fixed point
$d\in\U$. In both cases we call $d$ the DW point of $\phi$.

\begin{corollary}\label{contr-isom-DW}
Let $0<p<\infty$, $\phi:\U\to\U$ be analytic and $C_{\phi}$ the
induced composition operator on $H^{p}(\U)$. Then
$$
\n{C_{\phi}}\leq 1
$$
if and only if the DW point of $\phi$ is $\infty$.
\end{corollary}
\begin{proof}
The DW point of $\phi$ is $\infty$ if and only if the DW point of
its conjugate $\psi$ is $1$, which is equivalent to the conditions
$\psi(1)=1$ and $\psi'(1)\leq1$ (see \cite[Grand Iteration
Theorem, p. 78]{Sh}). By this and Lemma \ref{equality_of_angular
derivatives} the DW point of $\phi$ is $\infty$ if and only if
$\phi'(\infty)\geq 1$, which from \eqref{C estimate} is equivalent
to $\n{C_{\phi}}\leq 1$.
\end{proof}

\subsection{Semigroups of analytic self-maps of $\U$.}

Let $\{\phi_{t}:t\geq0\}$ be a semigroup of analytic self-maps of
$\U$. From \cite{BP}, the analytic function $G:\U\to\C$ given by
$$
G(z)=\lim_{t\to0}\frac{\partial \phi_{t}(z)}{\partial t}
$$
is the infinitesimal generator of $\{\phi_{t}\}$, characterizes
$\{\phi_{t}\}$ uniquely and satisfies
$$
G(\phi_t(z))=\frac{\partial \phi_{t}(z)}{\partial t}, \quad z\in
\U, \quad t\geq0.
$$
Suppose $\{\phi_{t}\}$ is not trivial, where $\{\phi_{t}\}$ is
called trivial if $\phi_{t}(z)\equiv z$ for all $t$, then it turns
out by \cite[Theorem 2.6]{BP} that all $\phi_{t}$, $t>0$, have a
common DW point d. Moreover if $d\neq\infty$, then $G(z)$ has the
unique representation
\begin{equation}\label{representation of G}
G(z)=F(z)(z-\overline{d})(z-d),
\end{equation} where $F:\U\to\C$
is analytic, $F\not\equiv0$, with $\textrm{Im}F\geq0$ on $\U$,
while if $d=\infty$, then $\textrm{Im}G\geq0$ and $G\not\equiv0$
on $\U$. The trivial semigroup has generator $G\equiv0$.

Likewise let $\widetilde{G}$ be the infinitesimal generator of the
conjugate semigroup $\{\psi_t\}$ of $\{\phi_{t}\}$. Then
$\widetilde{G}(\psi_t(w))=\frac{\partial\psi_t(w)}{\partial t}$,
$w\in\D$, and for each $z\in\U$,
$$
\widetilde{G}(\psi_t(\gamma^{-1}(z)))=\frac{\partial\psi_t(\gamma^{-1}(z))}{\partial
t}=\frac{\partial\gamma^{-1}(\phi_{t}(z))}{\partial
t}=\frac{2i}{(\phi_{t}(z)+i)^2}\frac{\partial\phi_{t}(z)}{\partial
t}.
$$
Hence letting $t$ tends to 0 we get
\begin{equation}\label{G}
\widetilde{G}(\gamma^{-1}(z))=\frac{2i}{(z+i)^2}G(z).
\end{equation}

\begin{proposition}
Let $\{\phi_{t}:t\geq0\}$ be a semigroup of analytic self-maps of
$\U$ with DW point $d$. Then we can classify $\{\phi_{t}\}$ as
follows.\\
1) If $d\in\U$, then there is a unique univalent function
$h:\U\to\C$ with $h(d)=0$, $h'(d)=1$ such that
    \begin{equation}\label{h - DW interior}
    \phi_t(z)=h^{-1}(e^{G'(d)t} h(z)), \quad z\in\U, \quad t\geq0.
    \end{equation}
2) If $d\in\partial\U=\R\cup\{\infty\}$, then there is a unique
univalent function $h:\U\to\C$ with $h(i)=0$, $h'(i)=1$ such that
    \begin{equation}\label{h - DW boundary}
    \phi_t(z)=h^{-1}(h(z)+G(i)t), \quad z\in\U, \quad t\geq0.
    \end{equation}
\end{proposition}
We call $h$ in either \eqref{h - DW interior} or \eqref{h - DW
boundary} the associated univalent function of $\{\phi_{t}\}$.
\begin{proof}
The above derived by the corresponding results about the
associated univalent function $k$ of the conjugate semigroup
$\{\psi_t\}$ shown in \cite[p. 234]{Si} and the observation that
$b=\gamma^{-1}(d)$ is the corresponding DW point of $\{\psi_t\}$.

Namely if $d\in\U$, i.e. $b\in\D$, then
$k(\psi_t(w))=e^{\widetilde{G}'(b)t}k(w)$, $w\in\D$, with $k(b)=0$
and $k'(b)=1$, from which
$$
k(\gamma^{-1}(\phi_{t}(z)))=e^{\widetilde{G}'(b)t}k(\gamma^{-1}(z)),
\quad z\in\U.
$$
By \eqref{G} we get $\widetilde{G}'(b)=G'(d)-\frac{2}{d+i}G(d)$
and the representation \eqref{representation of G} says that
$G(d)=0$. Hence $\widetilde{G}'(b)=G'(d)$ and \eqref{h - DW
interior} follows by setting
$h=\frac{(d+i)^2}{2i}k\circ\gamma^{-1}$.

If $d\in\partial\U$, i.e. $b\in\partial\D$, then
$k(\psi_t(w))=k(w)+\widetilde{G}(0)t,$ $w\in\D$, with $k(0)=0$ and
$k'(0)=1$, from which
$$
k(\gamma^{-1}(\phi_{t}(z)))=k(\gamma^{-1}(z))+\widetilde{G}(0)t
\quad z\in\U.
$$
Since by \eqref{G}
$\widetilde{G}(0)=\frac{2i}{(\gamma(0)+i)^2}G(i)=\frac{1}{2i}G(i)$,
setting $h=2ik\circ\gamma^{-1}$ \eqref{h - DW boundary} follows.
\end{proof}

\subsection{Composition semigroups on $H^p(\U)$}

Let $0 < p<\infty$. Given a semigroup $\{\phi_{t}\}$, the induced
semigroup $\{T_t\}$ of composition operators on $H^p(\U)$ does not
need always to consists of bounded operators, as the following
examples shows. However, each $T_t$ is bounded with $\n{T_t}\leq1$
if and only if the DW point of $\{\phi_{t}\}$ is $\infty$
(Corollary \ref{contr-isom-DW}).

\begin{example}
Consider the family of analytic functions
$$
\phi_t(z)=i\frac{z+i+e^{-t}(z-i)}{z+i-e^{-t}(z-i)}, \quad z\in\U,
\quad t\geq0.
$$
For each t, $\phi_t(z)=\gamma( e^{-t}\gamma^{-1}(z))$, that is
$\psi_t(w)=e^{-t}w, \ w\in \D,$ is the conjugate function of
$\phi_t$. From this it is clear that each $\phi_t$ maps $\U$ into
$\U$ and that $\{\phi_t\}$ is a semigroup. Also we have that
$$
\angle \lim_{z\to \infty}\phi_t(z)=i\frac{1+e^{-t}}{1-e^{-t}},
\quad t\geq0.
$$
So for $t>0$ we get that $\angle \lim_{z\to
\infty}\phi_t(z)<\infty$, thus $\phi_{t}'(\infty)=0$, which
implies that each $T_t, \ t>0$, is an unbounded operator on
$H^p(\U)$, $0 < p<\infty$.
\end{example}

\begin{example}
Consider the family of analytic functions
$$
\phi_t(z)=(z+1)^{e^{-t}}-1, \quad z\in\U, \quad t\geq0.
$$
For each $z\in\U$,
\begin{eqnarray*}
\mathrm{Im}\phi_t(z)&=&\mathrm{Im}e^{e^{-t}(\log{|z+1|}+i Arg(z+1))}\\
&=&e^{e^{-t}\log{|z+1|}}\mathrm{Im}e^{ie^{-t}Arg(z+1)}\\
&=&|z+1|^{e^{-t}}\sin(e^{-t}Arg(z+1))>0,
\end{eqnarray*}
so each $\phi_t$ maps $\U$ into $\U$. Also $\phi_0(z)=z$ and for
$t, s\geq0$
$$
\phi_t(\phi_s(z))=[(z+1)^{e^{-s}}-1+1]^{e^{-t}}-1=\phi_{t+s}(z),
$$
thus $\{\phi_{t}\}$ is a semigroup. Now we have that
$$
\phi_t'(z)=e^{-t}(z+1)^{e^{-t}-1}, \quad t\geq0.
$$
Thus $\phi_{t}'(\infty)=0$ for $t>0$ and so each $T_t, \ t>0$, is
an unbounded operator on $H^p(\U)$, $0 < p<\infty$.
\end{example}

\begin{theorem}\label{C_t bound-deriv}
Let $\{\phi_{t}:t\geq0\}$ be a semigroup of analytic self-maps of
$\U$ and $0<p<\infty$. Then the following are equivalent:
\begin{enumerate}
    \item For each $t>0$ the composition operator $T_t:f\mapsto f\circ\phi_t$ is bounded on $H^p(\U)$.
    \item There exists $t>0$ such that $T_t$ is bounded on $H^p(\U)$.
    \item There exists $t>0$ such that the angular derivative $\phi'_t(\infty)>0$.
    \item For each $t>0$ the angular derivative $\phi'_t(\infty)>0$.
\end{enumerate}
Moreover if one of the above assertions holds, then
\begin{equation}\label{C t estimate}
\n{T_t}=\phi_1'(\infty)^{-\frac{t}{p}}.
\end{equation}
\end{theorem}
\begin{proof}
We saw in \eqref{C estimate} that $T_t$ is bounded on $H^p(\U)$ if
and only if $\phi_{t}'(\infty)>0$, in which case
$$
\n{T_t}=\phi'_t(\infty)^{-\frac{1}{p}}.
$$
Suppose now there exists
$\phi_s, \ s>0,$ such that $\phi_{s}'(\infty)>0$. Then by Lemma
\ref{equality_of_angular derivatives} we get that for the
conjugate function $\psi_s$ on $\D$
$$
\psi_s(1)=1 \text{ and } \psi'_s(1)<\infty.
$$
It is known (see \cite[Theorems 1 and 5]{CMP1}, \cite{Co}) that
all members of the semigroup $\{\psi_{t}\}$ have common boundary
fixed points, that is $\psi_{t}(1)=1$ for each $t$. Furthermore,
since $\psi'_s(1)<\infty$ for some s, \cite[Lemmas 1 and 3]{CMP2}
say that
$$
\psi_t'(1)=\psi_1'(1)^t<\infty \text{\ \ for each } t\geq0,
$$
which implies that
$$
\phi_{t}'(\infty)=\phi_1'(\infty)^t>0
$$
for each $t$ and the conclusion follows.
\end{proof}

Moreover the property that $\{T_t\}$ consists of bounded operators
depends on the behavior of the infinitesimal generator $G(z)$ of
$\{\phi_{t}\}$ as $z\to\infty$ non-tangentially.

\begin{theorem}\label{C_t bound-G}
Let $\{\phi_{t}:t\geq0\}$ be a semigroup of analytic self-maps of
$\U$ with infinitesimal generator $G$ and $0<p<\infty$. Then the
following are equivalent:
\begin{enumerate}
    \item Each composition operator $T_t:f\mapsto f\circ\phi_t$ is bounded on $H^p(\U)$.
    \item The non-tangential limit
    $$
    \delta:=\angle\lim_{z\to\infty}\frac{G(z)}{z}
    $$
    exists finitely.
    \item The non-tangential limit
    $$
    \angle\lim_{z\to\infty}G'(z)
    $$
    exists finitely.
\end{enumerate}
Moreover if one of the above assertions holds, then \\
i) $\delta=\angle\lim_{z\to\infty}G'(z)\in\R$
and \\
ii) $\n{T_t}=e^{-\frac{\delta t}{p}}$ for each $t\geq0$.
\end{theorem}
\begin{proof}
$(1\Leftrightarrow 2)$. This is based on a theorem of M. D.
Contreras, S. D\'{i}az-Madrigal and Ch. Pommerenke in \cite{CMP2}.
Each $T_t$ is bounded if and only if $\phi'_t(\infty)>0$, that is
if and only if $\psi_t(1)=1$ and $\psi'_t(1)<\infty$ for each t.
The last, by \cite[Theorem 1]{CMP2}, is equivalent with the
finitely existence of
$\angle\lim_{w\to1}\frac{\widetilde{G}(w)}{1-w}$, where
$\widetilde{G}$ is the generator of the conjugate semigroup
$\{\psi_t\}$. As we shown in relation \eqref{G},
$$
\widetilde{G}(\gamma^{-1}(z))=\frac{2i}{(z+i)^2}G(z).
$$
From this and Lemma \ref{non-tang conv}
\begin{equation*}
\angle\lim_{w\to1}\frac{\widetilde{G}(w)}{1-w}=\angle\lim_{w\to1}\frac{G(\gamma(w))}{\gamma(w)+i}=
\angle\lim_{z\to\infty}\frac{G(z)}{z+i}=\angle\lim_{z\to\infty}\frac{G(z)}{z}
\end{equation*}
and the equivalence follows. Moreover then \cite[Theorem 1]{CMP2}
implies that $\delta\in\R$ and that $\psi'_t(1)=e^{-\delta t}$,
that is $\n{T_t}=e^{-\frac{\delta t}{p}}$ for each $t\geq0$.

$(2\Leftrightarrow 3)$. Suppose now $\delta$ exists finitely. Then
we can write
$$
G(z)=\delta z+h(z), \quad z\in\U
$$
where
\begin{equation}\label{h*}
\angle\lim_{z\to\infty}\frac{h(z)}{z}=0.
\end{equation}
Fix $z\in\U$. If r is small enough that
$\{z+re^{i\theta}:0\leq\theta\leq2\pi\}$ lies in $\U$, then by the
Cauchy integral formula we have
\begin{align*}
G'(z)&=\frac{1}{2\pi}\int_0^{2\pi}\frac{G(z+re^{i\theta})}{re^{i\theta}}\,d\theta\\
&=\frac{1}{2\pi}\int_0^{2\pi}\frac{\delta(z+re^{i\theta})}{re^{i\theta}}+\frac{h(z+re^{i\theta})}{re^{i\theta}}\,d\theta\\
&=\delta +
\frac{1}{2\pi}\int_0^{2\pi}\frac{h(z+re^{i\theta})}{z+re^{i\theta}}\frac{z+re^{i\theta}}{re^{i\theta}}\,d\theta.
\end{align*}
We show that the last integral tends to zero as $z\to\infty$
non-tangentially, thus $\angle\lim_{z\to\infty}G'(z)=\delta$. Fix
a non-tangential sector $T_u(\infty), \ u>0$ and let a sequence
$z_n\to\infty$ through $T_u(\infty)$. For each n choose $r_n$ to
be the distance of $z_n$ to the boundary of $T_{2u}(\infty)$. It
follows from \eqref{h*} that there is $M>0$ such that
$$
\Big|\frac{h(z_n+r_n e^{i\theta})}{z_n+r_n e^{i\theta}}\Big|<M,
\quad \text{for all } n \text{ and } \theta.
$$
Furthermore, let $\omega$ be the smallest angle made by the
boundary lines
 of $T_u(\infty)$ and $T_{2u}(\infty)$, then
$\frac{r_n}{|z_n|}>\sin{\omega}>0$ for each $n$ and so
$$
\Big|\frac{z_n+r_n e^{i\theta}}{r_n
e^{i\theta}}\Big|<\frac{|z_n|}{r_n}+1<\frac{1}{\sin{\omega}}+1.
$$
Therefore an application of Lebesgue's dominated convergence
theorem and \eqref{h*} gives that $G'(z_n)\to\delta$. Since $u$
and $\{z_n\}$ are arbitrary our conclusion follows.

Conversely suppose $\angle\lim_{z\to\infty}G'(z)$ exists finitely.
Fix $u>0$ and let $z_n\to\infty$ through the sector $T_u(\infty)$.
We can suppose that $\mathrm{Im}z_n>1$ for each $n$ and we can
write
\begin{align*}
\frac{G(z_n)-G(i)}{z_n-i}&=\frac{1}{z_n-i}\int_0^1\frac{\partial}{\partial
t}\Big(G((z_n-i)t+i)\Big)\,dt\\
&=\int_0^1 G'((z_n-i)t+i)\,dt.
\end{align*}
Since $\angle\lim_{z\to\infty}G'(z)<\infty$, there is $M>0$ such
that $|G'(z_n)|<M$ for all $n$ and applying Lebesgue's dominated
convergence theorem we get
$$
\lim_{n\to\infty}\frac{G(z_n)-G(i)}{z_n-i}=\int_0^1
\lim_{n\to\infty}G'((z_n-i)t+i)\,dt=\angle\lim_{z\to\infty}G'(z).
$$
Furthermore
\begin{align*}
\lim_{n\to\infty}\frac{G(z_n)-G(i)}{z_n-i}&=\lim_{n\to\infty}\frac{G(z_n)-G(i)}{z_n-i}+\lim_{n\to\infty}\frac{G(i)}{z_n-i}\\
&=\lim_{n\to\infty}\frac{G(z_n)}{z_n-i}=\lim_{n\to\infty}\frac{G(z_n)}{z_n-i}\frac{z_n-i}{z_n}\\
&=\lim_{n\to\infty}\frac{G(z_n)}{z_n}.
\end{align*}
Since $u$ and $\{z_n\}$ are arbitrary we get that
$\angle\lim_{z\to\infty}\frac{G(z)}{z}=\angle\lim_{z\to\infty}G'(z)<\infty$,
completing the proof.
\end{proof}

\section{Strong continuity and infinitesimal generator}

Throughout this section $\{\phi_{t}\}$ is a semigroup which
induces a semigroup $\{T_t\}$ of bounded composition operators on
$H^{p}(\U)$ spaces.

\begin{lemma}\label{(z+i)^l}
Let $0<p<\infty$ and  $\lambda\in \mathbb{C}$, then
$h_{\lambda}(z)=(z+i)^{\lambda}\in H^{p}(\U)$ if and only if
$\textrm{Re}\lambda <-\frac{1}{p}$.
\end{lemma}
\begin{proof}
Choosing a logarithmic branch we have
$\abs{(z+i)^{\lambda}}=e^{\textrm{Re}\lambda
\log|z+i|-\textrm{Im}\lambda\arg(z+i)}$ for each $z\in\U$ and we
can find constants $c, c'>0$ such that
$$
c'|z+i|^{\textrm{Re}\lambda}\leq\abs{(z+i)^{\lambda}}\leq
c|z+i|^{\textrm{Re}\lambda}.
$$
Thus we can suppose that $\lambda$ is real. Then
\begin{eqnarray*}
\n{h_{\lambda}}_{p}^{p}&=&\sup_{y>0}\int_{-\infty}^{+\infty}\Big(\frac{1}{x^{2}+(1+y)^2}\Big)^{-\frac{\lambda p}{2}}\,dx\\
&=&\sup_{y>0}(1+y)^{1+p\lambda}\int_{-\infty}^{+\infty}\Big(\frac{1}{x^{2}+1}\Big)^{-\frac{\lambda
p}{2}}\,dx.
\end{eqnarray*}
The last integral is convergent if and only if $-\frac{\lambda
p}{2}>\frac{1}{2}$, i.e. $\lambda <-\frac{1}{p}$, giving our
conclusion.
\end{proof}

The growth condition of $H^p(\U)$ functions in the following lemma
is known, see for instance \cite[p. 188]{Du}, \cite[p. 53]{Ga}.
Here we find the best possible constant.

\begin{lemma}
Let $0< p<\infty$ and suppose $f\in H^p(\U)$. Then for each $z\in
\U$,
\begin{equation}\label{growth estimate of H^p}
\abs{f(z)}^{p}\leq\frac{1}{4\pi}\frac{\n{f}_{p}^{p}}{\textrm{Im}z},
\end{equation}
where the constant $\frac{1}{4\pi}$ is the best possible.
\end{lemma}
\begin{proof}
We examine first the case $p=2$. We consider a point $z\in \U$ and
we recall that the $H^2(\U)$ function
$$
k_{z}(w)=\frac{i}{2\pi(w-\overline{z})}, \quad  w \in \U,
$$
is the reproducing kernel of $H^2(\U)$, that is for each $f\in
H^2(\U)$
$$
f(z)=\langle f, k_z\rangle,
$$
where the pairing is the inner product of $H^2(\U)$. From this
$$
\n{k_z}_2^2=\langle k_z, k_z\rangle=k_z(z)=\frac{1}{4\pi
\textrm{Im}z}.
$$
Let $f \in H^2(\U)$. Then Cauchy-Schwarz inequality implies that
\begin{equation*}
|f(z)|^2=|\langle f,k_{z} \rangle|^2 \leq
\n{f}_{2}^2\n{k_{z}}_{2}^2=\frac{\n{f}_{2}^2}{4\pi \textrm{Im}z}.
\end{equation*}
For $p\neq 2$, let $f \in H^p(\U)$ ($f\not\equiv 0$). By the
factorization of functions in $H^p(\U)$ (see \cite[Theorem
11.3]{Du})
$$
f(z)=b(z)g(z), \qquad z\in \U,
$$
where $g$ is a nonvanishing $H^p(\U)$ function such that
$$
|f(x)|=|g(x)| \ \text{ almost everywhere on $\R$ }
$$
and $b$ is a Blaschke product for the upper half plane of the form
$$
b(z)=\Big(\frac{z-i}{z+i}\Big)^m \prod_n
\frac{|z_n^2+1|}{z_n^2+1}\frac{z-z_n}{z-\overline{z_n}},
$$
where $m$ is a nonnegative integer and $z_n$ are the zeros
$(z_n\neq i)$ of $f$ in $\U$. Since
$\Big|\frac{z-a}{z-\overline{a}}\Big|<1$ for each $z, a\in\U$, it
follows that $|b(z)|< 1$, thus
$$
|f(z)|\leq |g(z)| \ \text{ for each $z\in\U$. }
$$
Moreover, since $g$ is a nonvanishing $H^p(\U)$ function, we can
choose a single-valued branch of $g^{p/2}$, which belongs to
$H^2(\U)$ and by the case $p=2$ follows that
\begin{align*}
|f(z)|^p&\leq |g^{p/2}(z)|^2\leq\frac{\n{g^{p/2}}_2^2}{4\pi\textrm{Im}z}\\
&=\frac{1}{4\pi\textrm{Im}z}\int_{-\infty}^{\infty}|g(x)|^p\,dx\\
&=\frac{1}{4\pi\textrm{Im}z}\int_{-\infty}^{\infty}|f(x)|^p\,dx\\
&=\frac{1}{4\pi}\frac{\n{f}_p^p}{\textrm{Im}z}.
\end{align*}

Finally, for each $0< p<\infty$ considering the estimate for the
$H^p(\U)$ function $h_{-2/p}(z)=(z+i)^{-\frac{2}{p}}$, for which
$\n{h_{-2/p}}_p^p=\pi$, we see that at $z=i$ equality holds, which
implies that $\frac{1}{4\pi}$ is the best possible constant for
this inequality.
\end{proof}

\begin{theorem}\label{Strong continuity and infinitesimal - Hp}
Suppose $1\leq p<\infty$ and let $\{\phi_t\}$ be a semigroup of
analytic self-maps of $\U$ which induces a semigroup $\{T_t\}$ of
bounded composition operators on $H^{p}(\U)$. Then
\begin{enumerate}
    \item $\{T_t\}$ is strongly continuous on $H^{p}(\U)$.
    \item If $G$ is the generator of $\{\phi_t\}$, then the infinitesimal
generator $\Gamma$ of $\{T_t\}$ has domain of definition
$$
D(\Gamma)=\{f\in H^{p}(\U): Gf'\in H^{p}(\U)\}
$$
and is given by
$$
\Gamma(f)=Gf', \quad  f\in D(\Gamma).
$$
\end{enumerate}
\end{theorem}

\begin{proof}
(1) For the strong continuity we need to show
$$
\lim_{t\to 0}\n{T_t(f)-f}_{p}=0,
$$
for every  $f\in H^{p}(\U)$. Fix a function $f\in H^{p}(\U)$.
Since the set of $H^p(\U)$ functions that are continuous on
$\U\cup\R$, denoted by $\mathcal{A}^p(\U)$, is dense in $H^p(\U)$
(see \cite[Corollary 3.3]{Ga}), for arbitrary $\epsilon>0$ we can
find $g\in \mathcal{A}^p(\U)$ such that $\n{f-g}_p<\epsilon$. Then
\begin{eqnarray*}
\n{T_t(f)-f}_p&\leq&
\n{T_t(f)-T_t(g)}_p+\n{T_t(g)-g}_p+\n{g-f}_p\\
&\leq&(\n{T_t}+1)\n{f-g}_p+\n{T_t(g)-g}_p
\end{eqnarray*}
and further by Theorem \ref{C_t bound-deriv} follows that
$$
\n{T_t(f)-f}_p\leq(\phi_1'(\infty)^{-\frac{t}{p}}+1)\epsilon+\n{T_t(g)-g}_p.
$$
Therefore, since $\phi_1'(\infty)^{-\frac{t}{p}}$ as a function of
$t$ is uniformly bounded on bounded intervals of $[0, \ +\infty)$,
we see that it suffices to show that for each $g\in
\mathcal{A}^p(\U)$
$$
\n{T_t(g)-g}_p=\n{g\circ\phi_{t}-g}_p\to 0, \quad \text{ as } t\to
0.
$$
By way of contradiction, suppose there is a function $g\in
\mathcal{A}^p(\U)$ and a sequence $\{t_n\}$ of values of $t$ such
that $t_n\to 0$ as $n\to \infty$ and
$$
\n{g\circ\phi_{t_n}-g}_p^p=\int_{-\infty}^{\infty}|g(\phi_{t_n}(x))-g(x)|^p\,dx\geq
s>0, \quad \text{ for each } n.
$$
We will show first that there is a subsequence $\{t_{n_k}\}$ such
that
\begin{equation}\label{p-w subsq}
g(\phi_{t_{n_k}}(x))\to g(x) \text{ almost everywhere on } \R.
\end{equation}
To do this we consider the $H^2(\U)$ function
$h(z)=\pi^{-\frac{1}{2}}(z+i)^{-1}$ for which
$$
\n{h}_2^2=\frac{1}{\pi}\int_{-\infty}^{\infty}\frac{1}{x^2+1}\,dx=1
$$
and we will show that $\n{h\circ\phi_{t_n}-h}_2\to 0$. If this is
true, then $h\circ\phi_{t_n}\to h$ in measure on $\R$ and thus
there is a subsequence $\{t_{n_k}\}$ such that
$h(\phi_{t_{n_k}}(x))\to h(x)$ almost everywhere on $\R$ (see for
instance \cite[Ex. 9, p. 85]{WZ}), from which
$$
\phi_{t_{n_k}}(x)\to x, \quad \text{ a. e. on } \R.
$$
Hence, since $g$ is continuous on $\U\cup\R$, \eqref{p-w subsq}
follows. Now to show that $\n{h\circ\phi_{t_n}-h}_2\to 0$ we use
the parallelogram law which asserts that
$$
\n{h\circ\phi_{t_n}-h}_2^2+\n{h\circ\phi_{t_n}+h}_2^2=2(\n{h\circ\phi_{t_n}}_2^2+\n{h}_2^2)
$$
for each $n$. From this and the norm of each $T_t$ in Theorem
\ref{C_t bound-deriv} follows
\begin{align*}
\n{h\circ\phi_{t_n}-h}_2^2&\leq2(\phi'_1(\infty)^{-t_n}+1)\n{h}_2^2-\n{h\circ\phi_{t_n}+h}_2^2\\
&=2(\phi'_1(\infty)^{-t_n}+1)-\n{h\circ\phi_{t_n}+h}_2^2.
\end{align*}
Further from the growth estimate \eqref{growth estimate of H^p}
$$
\n{h\circ\phi_{t_n}+h}^2_2\geq |h(\phi_{t_n}(i))+h(i)|^2 4\pi,
$$
thus
$$
\n{h\circ\phi_{t_n}-h}_2^2\leq
2(\phi'_1(\infty)^{-t_n}+1)-|h(\phi_{t_n}(i))+h(i)|^2 4\pi.
$$
Since $2(\phi'_1(\infty)^{-t_n}+1)\to 4$ and
$|h(\phi_{t_n}(i))+h(i)|^2 4\pi\to 4|h(i)|^2 4\pi=4$ as $n\to
\infty$, we get that
$$
\n{h\circ\phi_{t_n}-h}_2\to 0.
$$
Next we consider the sequence of functions
$$
S_{t_{n_k}}(x)=2^p\big(|g(\phi_{t_{n_k}}(x))|^p+|g(x)|^p\big)-|g(\phi_{t_{n_k}}(x))-g(x)|^p
$$
defined for almost all $x\in\R$. A standard inequality,
$(a+b)^p\leq 2^p(a^p+b^p)$, $a,b\geq0$, shows that the above
functions are nonnegative. Since as found above
$$
g(\phi_{t_{n_k}}(x))\to g(x) \text{ almost everywhere on } \R,
$$
an application of Fatou's lemma to the sequence $\{S_{t_{n_k}}\}$
gives
\begin{align*}
2^{p+1}\n{g}_p^p&\leq\liminf_{k\to
\infty}\int_{-\infty}^{\infty}2^p\big(|g(\phi_{t_{n_k}}(x))|^p+|g(x)|^p\big)-|g(\phi_{t_{n_k}}(x))-g(x)|^p\,dx\\
&=\liminf_{k\to
\infty}\Big[2^p\Big(\int_{-\infty}^{\infty}|g(\phi_{t_{n_k}}(x))|^p\,dx+\int_{-\infty}^{\infty}|g(x)|^p\,dx\Big)\\
&\qquad\qquad\qquad\qquad\qquad\qquad\qquad\quad-\int_{-\infty}^{\infty}|g(\phi_{t_{n_k}}(x))-g(x)|^p\,dx\Big]\\
&\leq \liminf_{k\to
\infty}\Big[2^p(\phi'_1(\infty)^{-t_{n_k}}+1)\n{g}_p^p-\int_{-\infty}^{\infty}|g(\phi_{t_{n_k}}(x))-g(x)|^p\,dx\Big]\\
&=2^{p+1}\n{g}_p^p-\limsup_{k\to
\infty}\int_{-\infty}^{\infty}|g(\phi_{t_{n_k}}(x))-g(x)|^p\,dx.
\end{align*}
Thus
$$
0\geq \limsup_{k\to
\infty}\int_{-\infty}^{\infty}|g(\phi_{t_{n_k}}(x))-g(x)|^p\,dx\geq
\liminf_{k\to
\infty}\int_{-\infty}^{\infty}|g(\phi_{t_{n_k}}(x))-g(x)|^p\,dx\geq0
$$
and so we conclude that
$$
\n{g\circ\phi_{t_{n_k}}-g}_p\to 0,
$$
which contradicts the original choice of $g$ and $\{t_n\}$.

(2) By definition the domain $D(\Gamma)$ of $\Gamma$ consists of
all $f\in H^{p}(\U)$ for which the limit $\lim_{t \to
0}\frac{T_t(f)-f}{t}$ exists in $H^{p}(\U)$ and
$$
\Gamma(f)=\lim_{t \to 0}\frac{T_t(f)-f}{t}, \quad f\in D(\Gamma).
$$
The growth estimate \eqref{growth estimate of H^p} shows that
convergence in the norm of $H^p(\U)$ implies in particular
pointwise convergence, therefore for $f\in D(\Gamma)$,
\begin{align*}
\Gamma(f)(z)&=\lim_{t \to
0}\frac{T_t(f)(z)-f(z)}{t}=\frac{\partial f(\phi_{t}(z))}{\partial
t}\biggm|_{t=0}\\
&=f'(z)\frac{\partial\phi_{t}(z)}{\partial
t}\biggm|_{t=0}=G(z)f'(z).
\end{align*}
This shows that $D(\Gamma)\subseteq \{f\in H^{p}(\U): Gf'\in
H^{p}(\U)\}$. Conversely let $f\in H^{p}(\U)$ such that $Gf'\in
H^{p}(\U)$. Then for $z\in\U$,
\begin{eqnarray*}
T_t(f)(z)-f(z)&=&\int_{0}^{t}\frac{\partial
f(\phi_{s}(z))}{\partial
s}\,ds\\
&=&\int_{0}^{t}\frac{\partial\phi_{s}(z)}{\partial
s}f'(\phi_{s}(z))\,ds\\
&=&\int_{0}^{t}G(\phi_{s}(z))f'(\phi_{s}(z))\,ds.
\end{eqnarray*}
Therefore
$$
\frac{T_t(f)-f}{t}=\frac{1}{t}\int_{0}^{t}T_s(Gf')\,ds.
$$
Since $\{T_t\}$ is strongly continuous the latter tends in the
norm of $H^p(\U)$ to $Gf'$ as $t\to0$ (see \cite[Theorem 2.4, p.
4]{Pa}). Thus $f\in D(\Gamma)$, completing the proof.
\end{proof}

\begin{corollary}\label{uniformly continuity}
Suppose $1\leq p<\infty$. The only uniformly continuous semigroup
$\{T_t\}$ on $H^p(\U)$ is the one induced by the trivial
semigroup.
\end{corollary}
\begin{proof}
Suppose a semigroup $\{\phi_{t}\}$ with generator $G$ induces a
semigroup $\{T_t\}$ which is continuous in the uniform operator
topology. Then the infinitesimal generator $\Gamma$ of $\{T_t\}$
is bounded on $H^p(\U)$. Thus for each $f\in H^p(\U)$ by Theorem
\ref{Strong continuity and infinitesimal - Hp} we have that
$\Gamma(f)=Gf'\in H^p(\U)$ and moreover
$$
\n{Gf'}_{p}\leq\n{\Gamma} \n{f}_{p}.
$$
Consider now for each $n$ natural the analytic functions
$$
e_{n}(z)=\frac{1}{\pi^{1/p}}\frac{(\gamma^{-1}(z))^n}{(z+i)^{2/p}},
\quad z\in\U,
$$
where we recall that $\gamma^{-1}(z)=\frac{z-i}{z+i}:\U\to\D$. By
Lemma \ref{(z+i)^l} we see that $e_{n}\in H^p(\U)$ and furthermore
$$
\n{e_{n}}^p_{p}=\frac{1}{\pi}\int_{-\infty}^{\infty}\frac{|\gamma^{-1}(x)|^{pn}}{|x+i|^2}\,dx
=\frac{1}{\pi}\int_{-\infty}^{\infty}\frac{1}{x^2+1}\,dx=1.
$$
Thus $\n{Ge'_{n}}_{p}\leq\n{\Gamma}<\infty$ for each $n$. A short
computation gives
$$
e'_{n}(z)=\frac{1}{\pi^{1/p}}\frac{(\gamma^{-1}(z))^{n-1}}{(z+i)^{\frac{2}{p}+2}}\big[-\frac{2}{p}z+(2n+\frac{2}{p})i\big].
$$
Also let $\omega(z)=-\frac{p}{p+2}(z+i)^{-\frac{2}{p}-1}$ for
which $\omega'(z)=(z+i)^{-\frac{2}{p}-2}$. Since from Lemma
\ref{(z+i)^l} $\omega\in H^p(\U)$ we get that $G\omega'\in
H^p(\U)$. Therefore
\begin{align*}
\n{Ge'_{n}}_p^p&=\frac{1}{\pi}\int_{-\infty}^\infty|(G\omega')(x)|^p|(\gamma^{-1}(x))^{n-1}[-\frac{2}{p}x+(2n+\frac{2}{p})i]|^p\,dx\\
&\geq\frac{n^p}{\pi}\int_{-\infty}^\infty|(G\omega')(x)|^p\,dx=
\frac{n^p}{\pi}\n{G\omega'}_p^p
\end{align*}
and it follows that for each $n$,
$\frac{n}{\pi^{1/p}}\n{G\omega'}_p\leq\n{\Gamma}<\infty$. Thus we
conclude that $G\equiv 0$, that is $\{\phi_{t}\}$ is trivial,
which obviously induces an uniformly continuous semigroup
$\{T_t\}$.
\end{proof}

\begin{proposition}\label{point spectrum of infinitesimal - Hp}
Suppose $1\leq p<\infty$ and let $\{\phi_{t}\}$ be a semigroup
which induces a semigroup $\{T_t\}$ of bounded operators on
$H^{p}(\U)$. If $G$ is the generator, $d$ the DW point and $h$ the
associated univalent function of $\{\phi_{t}\}$, then we have the
following for $\sigma_\pi(\Gamma)$, the point spectrum of the generator $\Gamma$ of $\{T_t\}$. \\
i) If $d\in\U$, then
$$
\sigma_\pi(\Gamma)=\{G'(d)k : h(z)^k\in H^p(\U), \ k=0, 1,
2,...\}.
$$
ii) If $d\in\partial\U$, then
$$
\sigma_\pi(\Gamma)=\{G(i)\nu\in\C : e^{\nu h(z)}\in H^p(\U)\}.
$$
\end{proposition}
\begin{proof}
We have to solve $\Gamma(f)=\lambda f$ for $\lambda\in\C$ and
$f\in D(\Gamma)$, $f\not\equiv0$. This by Theorem \ref{Strong
continuity and infinitesimal - Hp} is equivalent to the
differential equation
\begin{equation}\label{differential equation for point spectrum}
G(z)f'(z)=\lambda f(z), \quad f\in H^p(\U), \quad f\not\equiv0.
\end{equation}

i) Suppose $d\in\U$. Then $\phi_t(z)=h^{-1}(e^{G'(d)t} h(z))$ for
each $t$ with $h(d)=0$ and $h'(d)=1$ (see \eqref{h - DW
interior}). So we get that
$$
G(z)=\frac{\partial\phi_t(z)}{\partial t}\Big|_{t=0}
=G'(d)\frac{e^{G'(d)t}h(z)}{h'(\phi_t(z))}\Big|_{t=0}=G'(d)\frac{h(z)}{h'(z)}.
$$
Notice that $G'(d)\neq0$, since otherwise $\{\phi_{t}\}$ will be
trivial. Thus \eqref{differential equation for point spectrum}
becomes
$$
\frac{h(z)}{h'(z)}f'(z)=\frac{\lambda}{G'(d)} f(z), \quad f\in
H^p(\U), \quad f\not\equiv0.
$$
If a $\lambda\in\C$ and a function $f$ satisfy this, choosing
$r\in (0,\ \textrm{Im}d)$ such that $f(z)$ has no zeros on
$|z-d|=r$, we get
$$
\frac{1}{2\pi
i}\int_{|z-d|=r}\frac{f'(\zeta)}{f(\zeta)}\,d\zeta=\frac{\lambda}{G'(d)}\frac{1}{2\pi
i}\int_{|z-d|=r}\frac{h'(\zeta)}{h(\zeta)}\,d\zeta
$$
and by the argument principle follows that $\lambda=G'(d)k$, where
$k$ is a nonnegative integer. Since the nonzero analytic solutions
of
$$
\frac{h(z)}{h'(z)}f'(z)=k f(z)
$$
are of the form $ch(z)^{k}, \ c\neq 0,$ the conclusion follows.

ii) Suppose $d\in\partial\U$. Then $\phi_t(z)=h^{-1}(h(z)+G(i)t)$
for each $t$ (see \eqref{h - DW boundary}) and
$$
G(z)=\frac{\partial\phi_t(z)}{\partial t}\Big|_{t=0}
=\frac{G(i)}{h'(\phi_t(z))}\Big|_{t=0}=\frac{G(i)}{h'(z)},
$$
where, since $\{\phi_{t}\}$ is not trivial, $G(i)\neq0$. Thus by
\eqref{differential equation for point spectrum} we see that
$G(i)\nu\in \sigma_\pi(\Gamma)$ if and only if there exists a
function $f\in H^p(\U)$, $f\not\equiv0$, that satisfies
$$
\frac{f'(z)}{h'(z)}=\nu f(z).
$$
Since the nonzero analytic solutions of this are of the form
$ce^{\nu h(z)}, \ c\neq0,$ the conclusion follows.
\end{proof}

\section*{Acknowledgement}
This article constitutes part of my Ph.D. thesis written under the
supervision of Professor A. G. Siskakis at the Aristotle
University of Thessaloniki.  I would like to thank him for
suggesting this subject and for valuable discussions on the topic
of semigroups of composition operators.

\bibliographystyle{amsplain}

\end{document}